\documentclass[a4paper,12pt]{article}

\input{mymacros.sty}

\newcommand{\grG}{\gr^G}
\newcommand{\qgrG}{\qgr^G}

\title{On graded stable derived categories of isolated Gorenstein
quotient singularities}
\author{Kazushi Ueda}
\date{}
\begin{document}

\maketitle

\begin{abstract}
We show the existence
of a full exceptional collection
in the graded stable derived category
of a Gorenstein isolated quotient singularity
using a result of Orlov \cite{Orlov_DCCSTCS}.
We also show that
the equivariant graded stable derived category
of a Gorenstein Veronese subring of a polynomial ring
with respect to an action of a finite group
has a full strong exceptional collection,
even if the corresponding quotient singularity
is neither isolated nor Gorenstein.
\end{abstract}

\section{Introduction}

Let $A = \bigoplus_{d=0}^\infty A_d$ be an $\bN$-graded
Noetherian ring over a field $k$.
%We assume that the characteristic of $k$ does not divide
%the order of finite groups appearing in this paper.
The ring $A$ is said to be {\em connected} if $A_0 = k$.
A connected ring $A$ is {\em Gorenstein} with parameter $a$
if $A$ has finite injective dimension as a right module over itself and
$$
 \RHom_A(k, A) = k(a)[-n].
$$
%see e.g. \cite[Definition 4.13(ii)]{Yekutieli_DCNGA}.
Here $\bullet(a)$ denotes the shift of $\bZ$-grading and
$\bullet[-n]$ is the shift in the derived category.
%$M(k)_d = M_{k+d}$ for a graded module $M$
%and any integer $k$.
%Let $D^b(\gr A)$ be the bounded derived category of
%finitely-generated $\bZ$-graded right $A$-modules
%and $D^\perf(\gr A)$ be its full subcategory
%consisting of perfect complexes.
%Here, an object of $D^b(\gr A)$ is {\em perfect}
%if it is quasi-isomorphic
%to a bounded complex of projective modules.
The {\em graded stable derived category} is
the quotient category
\begin{align} \label{eq:dbsing}
 \Dbsing(\gr A) = D^b(\gr A) / D^\perf(\gr A)
\end{align}
of the bounded derived category $D^b(\gr A)$
of finitely-generated  $\bZ$-graded right $A$-modules
by the full triangulated subcategory $D^\perf(\gr A)$
consisting of perfect complexes.
Here, an object of $D^b(\gr A)$ is {\em perfect}
if it is quasi-isomorphic
to a bounded complex of projective modules.
Stable derived categories are introduced
by Buchweitz \cite{Buchweitz_MCM}
motivated by the theory of {\em matrix factorizations}
by Eisenbud \cite{Eisenbud_HACI}.
%The stable derived category of a Gorenstein ring
%is equivalent to the stable category of graded Cohen-Macaulay modules,
%which has been studied in ring theory for decades.
Stable derived categories are also known as
{\em triangulated categories of singularities},
introduced by Orlov \cite{Orlov_TCS}
based on an idea of Kontsevich
to study B-branes on Landau-Ginzburg models.

Let $R = k[x_1,\ldots,x_{n+1}]$ be a polynomial ring
in $n+1$ variables over a field $k$.
We equip $R$ with a $\bZ$-grading such that
$\deg x_i = 1$ for all $i$.
Let $G$ be a finite subgroup of $SL_{n+1}(k)$
whose order is not divisible by the characteristic of $k$.
Assume that the natural action of $G$
on the affine space $\bA^{n+1} = \Spec R$
is free outside of the origin.
This assumption is equivalent to the condition
that the invariant subring $A = R^G$ has
an isolated singularity at the origin
\cite[Corollary 8.2]{Iyama-Yoshino_MTC}.
Two examples of
the stable derived categories of $A$ are studied
by Iyama and Yoshino \cite{Iyama-Yoshino_MTC}
and Keller, Murfet and Van den Bergh
\cite{Keller-Murfet-Van_den_Bergh}.
%using a result of Orlov \cite{Orlov_DCCSTCS},
The general case is studied by Iyama and Takahashi
\cite{Iyama-Takahashi_TCTQS}.
%who has shown that $\Dbsing(\gr A)$ has
%a full strong exceptional collection.

Let $d \in \bN$ be a divisor of $n+1$ and
$
 B = \bigoplus_{i=0}^\infty A_{i d}
$
be the $d$-th Veronese subring of $A$.
We prove the following in this paper:

\begin{theorem} \label{th:main1}
The stable derived category $\Dbsing(\gr B)$
has a full exceptional collection.
\end{theorem}

The full exceptional collection
given in Theorem \ref{th:main1}
is strong when $d = n+1$.
One the other hand,
a result of Iyama and Takahashi
\cite[Theorem 1.7]{Iyama-Takahashi_TCTQS}
gives a full strong exceptional collection
%in $\Dbsing(\gr B)$
for $d = 1$.
The proof of Theorem \ref{th:main1} is based
on the existence of a full strong exceptional collection
in the derived category of coherent sheaves on the stack
$\Proj B = [(\Spec B \setminus \bszero) / \bG_m]$
and a result of Orlov
\cite[Theorem 2.5.(i)]{Orlov_DCCSTCS}.
%Along the proof of Theorem \ref{th:main1},
%we also show the following:
%
%\begin{theorem} \label{th:PmodG}
%Let $G$ be a finite subgroup of $\PGL_{n+1}(k)$
%whose order is not divisible by the characteristic of $k$.
%Then the derived category $D^b \coh^G \bP^n$
%of $G$-equivariant coherent sheaves on $\bP^n$
%has a full strong exceptional collection.
%\end{theorem}

Next we discuss equivariant graded stable derived categories.
%not of the invariant subring but
%of the crossed product ring.
%Note that the crossed product ring corresponds
%to the stack-theoretic quotient,
%whereas the invariant ring corresponds
%to the scheme-theoretic quotient.
Let $A$ be an $\bN$-graded connected Gorenstein ring
with parameter $a > 0$
and $G$ be a finite group acting on $A$
whose order is not divisible by the characteristic of $k$.
The {\em crossed product algebra} $A \rtimes G$ is the vector space
$A \otimes k[G]$ equipped with the ring structure
$$
 (a_1 \otimes g_1) \cdot (a_2 \otimes g_2)
  = a_1 \cdot g_1 (a_2) \otimes g_1 \circ g_2,
$$
where $k[G]$ is the group ring of $G$.
%$a_i \in A$ and $g_i \in G$.
%If $A$ is commutative, then
%the category $\module A^G$ of finitely-generated modules
%over the invariant subring $A^G$ is
%the category of coherent sheaves
%on the scheme-theoretic quotient $\Spec A / G = \Spec A^G$,
%whereas
A right $A \rtimes G$-module is often called
a {\em $G$-equivariant $A$-module}.
%The category $\module A \rtimes G$
%of finitely-generated right $A \rtimes G$-modules is
%the category of coherent sheaves
%on the quotient stack $[\Spec A / G]$.
%The latter category is also equivalent to
%the category of $G$-equivariant coherent sheaves on $\Spec A$.
The crossed product algebra $A \rtimes G$ inherits a grading from $A$
so that the degree zero part is given by the group ring;
$(A \rtimes G)_0 = k[G]$.
This graded ring is not connected
if $G$ is non-trivial.

Let $\gr^G A$ be the abelian category
of finitely-generated $\bZ$-graded right $A \rtimes G$-modules and
$\tor^G A$ be its Serre subcategory
consisting of finite-dimensional modules.
The quotient abelian category is denoted by
$$
 \qgr^G A = \gr^G A / \tor^G A.
$$
If $A$ is commutative,
then $\qgr^G A$ is equivalent
to the abelian category $\coh^G (\Proj A)$
of $G$-equivariant coherent sheaves
of the stack
$
 \Proj A = [(\Spec A \setminus \bszero) / \bG_m].
$
Let $\Irrep(G) = \{ \rho_0, \ldots, \rho_r \}$
be the set of irreducible representations of $G$
where $\rho_0$ is the trivial representation.
%and $A \otimes \rho_i$ be the right $A \rtimes G$-module
%whose $A$-module structure is the same as
%$A^{\oplus \dim \rho_i}$ and
%the $k[G]$-module structure is given by $A \otimes \rho_i$.
For any $k \in \bZ$, the image of the graded $A \rtimes G$-module
$A(k) \otimes \rho_i$
by the projection functor
$\pi : \gr A \rtimes G \to \qgr^G A$
will be denoted by
$
 \scO(k) \otimes \rho_i.
 % = \pi(A(k) \otimes \rho)
 $
The following is a straightforward generalization of
\cite[Theorem 2.5.(i)]{Orlov_DCCSTCS}:

\begin{theorem} \label{th:main3}
There is a full and faithful functor
$
 \Phi : \Dbsing(\gr^G A) \to D^b (\qgr^G A)
$
and a semiorthogonal decomposition
\begin{align*}
 D^b(\qgr^G A) = \langle \scO \otimes \rho_0, &\ldots,
  \scO \otimes \rho_r,
  \scO(1) \otimes \rho_0, \ldots,
  \scO(1) \otimes \rho_r, \ldots,
   \\
  &\scO(a-1) \otimes \rho_0, \ldots, \scO(a-1) \otimes \rho_r,
  \Phi \Dbsing(\gr^G A) \rangle.
\end{align*}
\end{theorem}

%We use Theorem \ref{th:main3}
%to study the stable derived category
%of a crossed product algebra
%of a Veronese ring
%of a polynomial ring
%with any finite subgroup of $\GL_{n+1}(k)$.
Let
$
 R = k[x_1, \ldots, x_{n+1}]
$
be a polynomial ring in $n+1$ variables and
$
 A = \bigoplus_{i=0}^\infty R_{i d}
$
be the $d$-th Veronese subring.
We assume that $d$ is a divisor of $n+1$
so that $A$ is Gorenstein with parameter $a = (n+1)/d$.
Let $G$ be any finite subgroup of $\GL_{n+1}(k)$
whose order is not divisible by the characteristic of $k$.
We have the following corollary of
Theorem \ref{th:main3}:
%We also show the following in this paper:

\begin{theorem} \label{th:main2}
The stable derived category $\Dbsing(\gr^G A)$
has a full strong exceptional collection.
\end{theorem}

The organization of this paper is as follows:
In Section \ref{sc:invariant},
we study $\Proj B$
for the Veronese subring $B$ of the invariant ring
%$k[x_1, \ldots, x_{n+1}]^G$
and prove Theorems \ref{th:main1}.
We prove Theorem \ref{th:main3}
in Section \ref{sc:crossed_product},
which immediately gives Theorem \ref{th:main2}.
We discuss a few examples in Section \ref{sc:example}.

{\bf Acknowledgment}:
I thank Akira Ishii for valuable discussions.
This work is supported by Grant-in-Aid for Young Scientists (No.20740037).

\section{Invariant subrings} \label{sc:invariant}

Let $\scD$ be a triangulated category and
$\scN \subset \scD$ be a full triangulated subcategory.
The {\em right orthogonal} to $\scN$ is the full
subcategory $\scN^\bot \subset \scD$ consisting of objects
$M$ such that $\Hom(N, M) = 0$
for any $N \in \scN$.
The left orthogonal $\!^\bot \scN$ is defined
similarly by $\Hom(M, N) = 0$
for any $N \in \scN$.
A full triangulated subcategory $\scN$
of a triangulated category $\scD$ is {\em left admissible}
if any $X \in \scD$ sits inside a distinguished triangle
$N \to X \to M \xto{[1]} N$
such that $N \in \scN$ and $M \in \scN^\bot$.
Right admissible subcategories are defined similarly.
A sequence $(\scN_1, \ldots, \scN_n)$ of full triangulated
subcategories is a {\em weak semiorthogonal decomposition}
if there is a sequence
$
 \scN_1 = \scD_1 \subset \scD_2 \subset \cdots \subset \scD_n = \scD
$
of left admissible subcategories
such that $\scN_p$ is left orthogonal to $\scD_{p-1}$
in $\scD_p$.
The decomposition is {\em orthogonal}
if $\Hom(N, M) = 0$ for
any $N \in \scN_i$ and $M \in \scN_j$
with $i \ne j$.

Let $k$ be a field and
$\scD$ be a $k$-linear triangulated category.
An object $E$ of $\scD$ is {\em exceptional} if
$\Hom(E, E)$ is spanned by the identity morphism and
$\Ext^i(E, E) = 0$ for $i \ne 0$.
A sequence $(E_1, \ldots, E_r)$ of exceptional objects
is an {\em exceptional collection}
if $\Ext^i(E_j, E_\ell) = 0$ for any $i$ and any $1 \le \ell < j \le r$.
An exceptional collection is {\em strong}
if $\Ext^i(E_j, E_\ell) = 0$ for any $i \ne 0$ and any $1 \le j \le \ell \le r$.
An exceptional collection is {\em full}
if the smallest full triangulated subcategory
of $\scD$ containing it is the whole of $\scD$.

Let $G$ be a finite subgroup of $SL_{n+1}(k)$
acting freely on $\bA^{n+1} \setminus \bszero$.
We assume that
the order of $G$ is not divisible by the characteristic
of the base field $k$.
The set of irreducible representations of $G$
will be denoted by $\Irrep(G) = \{ \rho_0, \ldots, \rho_r \}$
where $\rho_0$ is the trivial representation.
Let further $R = k[x_1, \ldots, x_{n+1}]$
be the coordinate ring of $\bA^{n+1}$ and
$A = R^G$ be the invariant subring.
Equip $R$ with the $\bN$-grading
such that $\deg x_i = 1$ for all $i = 1, \ldots, n+1$,
which induces an $\bN$-grading on $A$.
This defines a $\bG_m$-action on $\Spec A$,
and let
$$
 Y := \Proj A
  = [(\Spec A \setminus \bszero) / \bG_m]
  = [((\Spec R \setminus \bszero) / G) / \bG_m]
$$
be the quotient stack.
%Since $G$-action on $\bA^{n+1} \setminus \bszero$ is free
%and commutes with the $\bG_m$-action,
%one has
%%$$
%% \Spec A \setminus \bszero = [(\bA^{n+1} \setminus \bszero) / G]
%%$$
%%and
%$
% Y = [\bP^n / G]
%$
%where $\bP^n$ is the projective space.
The abelian category $\coh^G \bP^n$
of $G$-equivariant coherent sheaves on $\bP^n$
is equivalent to
the abelian category $\coh Y$
of coherent sheaves on $Y$,
which in turn is equivalent to the quotient category
$$
 \qgr A = \gr A / \tor A
$$
of the abelian category
$\gr A$ of finitely-generated $\bZ$-graded $A$-modules
by the Serre subcategory
consisting of finite-dimensional modules
\cite[Proposition 2.17]{Orlov_DCCSTCS}.
Note that $G$-action on $\bP^n$ may not be free.

The following theorem is due to Beilinson:

\begin{theorem}[{Beilinson \cite{Beilinson}}]
 \label{th:Beilinson}
$D^b \coh \bP^n$ has a full strong exceptional collection
$$
 (\scO_{\bP^n}, \scO_{\bP^n}(1), \ldots, \scO_{\bP^n}(n))
$$
consisting of line bundles.
\end{theorem}

As an immediate corollary to Theorem \ref{th:Beilinson},
we have the following:

\begin{corollary} \label{cr:PmodG}
$D^b \coh^G \bP^n$ has a full strong exceptional collection
$$
 (\scO_{\bP^n} \otimes \rho_0, \ldots, \scO_{\bP^n} \otimes \rho_r,
  \scO_{\bP^n}(1) \otimes \rho_0, \ldots, \scO_{\bP^n}(1) \otimes \rho_r,
   \ldots, \scO_{\bP^n}(n) \otimes \rho_0, \ldots, \scO_{\bP^n}(n) \otimes \rho_r).
$$
\end{corollary}

Let $d$ be a divisor of $n+1$ and
$
 B = \bigoplus_{i \in \bZ} A_{i d}
$
be the $d$-th Veronese subring of $A = R^G$.
Let further
$
 G_d = G / T_d
$
be the quotient of $G$
by the diagonal subgroup
$
 T_d = \{ \zeta \cdot \id_{\bA^{n+1}} \in G
  \mid \zeta^d = 1 \}
$
consisting of $d$-th roots of unity.
Then $B$ is the invariant subring
$(R^{(d)})^{G_d}$
of the $d$-th Veronese subring $R^{(d)}$
of $R = k[x_1, \ldots, x_{n+1}]$,
and one has
$$
 X := \Proj B = [((\Spec R^{(d)} \setminus \bszero) / G_d) / \bG_m].
$$
The group $T_d$ is a cyclic group
whose order $e$ is a divisor of $d$.
If $T_d$ is non-trivial,
then $G_d$ is not a subgroup of $\SL_{n+1}(k)$
but a subgroup of its quotient $\SL_{n+1}(k) / T_d$,
and the line bundle $\scO_{\bP^n}(1)$ does not have
a $G_d$-linearization.
On the other hand,
the line bundle $\scO_{\bP^n}(e)$ does have a $G_d$-linearization
and descends to a line bundle $\scO_X(e)$
on $X$.

Recall that the {\em root stack} $\sqrt[e]{\scL/X}$
of a line bundle $\scL$ on a stack $X$ is the stack
whose object over $\varphi : T \to X$ is
a line bundle $\scM$ on $T$
together with an isomorphism
$
 \scM^{\otimes e} \simto \varphi^* \scL
$
\cite{Abramovich-Graber-Vistoli, Cadman_US}.
The morphism $G \to G_d$ of finite groups induces a morphism
$p : Y \to X$ of quotient stacks,
and the isomorphism
\begin{equation*} %\label{eq:root1}
 \phi : \scO_Y(1)^{\otimes e} \simto \pi^* \scO_X(e)
\end{equation*}
of line bundles gives an identification of
$Y$ with the root stack
$
 \sqrt[e]{\scO_X(e) / X}.
$
%with the root stack.
It follows that
there is an orthogonal decomposition
\begin{equation} \label{eq:dbcohY}
 D^b \coh Y = 
  \la
   p^* D^b \coh X, \ 
   \scO_Y(1) \otimes p^* D^b \coh X,
   \cdots,
   \scO_Y(e-1) \otimes p^* D^b \coh X
  \ra
\end{equation}
of the derived category
%of coherent sheaves on $Y$
\cite[Lemma 4.1]{Ishii-Ueda_SMEC}.

The invariant ring $A$ is Gorenstein
with parameter
$
 \deg x_1 + \cdots + \deg x_{n+1}
  = n+1
$
by Watanabe \cite[Theorem 1]{Watanabe_CISG},
and its Veronese subring $B$ is Gorenstein
with parameter $a = (n+1)/d$
by Goto and Watanabe
\cite[Corollary 3.1.5]{Goto-Watanabe_GRI}.
The following theorem is due to Orlov:

\begin{theorem}[{\cite[Theorem 2.5.(i)]{Orlov_DCCSTCS}}]
 \label{th:Orlov}
If $B$ is a Gorenstein ring with parameter $a > 0$,
then there is a full and faithful functor
$
 \Phi : \Dbsing(\gr B) \to D^b (\qgr B)
$
and a semiorthogonal decomposition
\begin{align*}
 D^b(\qgr B) = \langle \pi B, \ldots, \pi (B(a-1)),
  \Phi \Dbsing(\gr B) \rangle,
\end{align*}
where $\pi : \gr B \to \qgr B$ is the natural projection functor.
\end{theorem}

Now we prove Theorem \ref{th:main1}.
First consider the case $d = 1$.
Recall that the {\em right mutation}
of an exceptional collection is given by
$$
 (E, F)
  \mapsto (F, R_F E)
$$
where $R_F E$ is the mapping cone
$$
 R_F E = \{ E \to \hom(E, F)^\vee \otimes F \}.
$$
See \cite{Rudakov} and references therein
for more about mutations of exceptional collections.
Write $E_{i,j} = \scO_Y(i) \otimes \rho_j$ and
perform successive right mutations
\begin{align*}
 (E_{0,0}, &\ldots, E_{0,r},
  E_{1,0}, \ldots, E_{1,r},
  \ldots,
  E_{n,0}, \ldots, E_{n,r}) \\
 &\mapsto
 (E_{0,0}, \ldots, E_{0,r-1},
  E_{1,0}, R_{E_{1,0}} E_{0,r},
  E_{1,1}, \ldots, E_{1,r},
  \ldots,
  E_{n,0}, \ldots, E_{n,r}) \\
 &\mapsto
 (E_{0,0}, \ldots, E_{0,r-2},
  E_{1,0},
  R_{E_{1,0}} E_{0,r-1}, R_{E_{1,0}} E_{0,r},
  E_{1,1}, \ldots, E_{1,r},
  \ldots,
  E_{n,0}, \ldots, E_{n,r}) \\
 &\mapsto \cdots \\
 &\mapsto
 (E_{0,0}, E_{1,0}, R_{E_{1,0}} E_{0,1}, \ldots,
  R_{E_{1,0}} E_{0,r},
  E_{1,1}, \ldots, E_{1,r},
  \ldots,
  E_{n,0}, \ldots, E_{n,r}) \\
 &\mapsto \cdots \\
 &\mapsto
 (E_{0,0}, E_{1,0}, E_{2,0},
  R_{E_{2,0}} R_{E_{1,0}} E_{0,1}, \ldots,
  R_{E_{2,0}} R_{E_{1,0}} E_{0,r}, \\
  &\ \hspace{5cm}
  R_{E_{2,0}} E_{1,1}, \ldots,
  R_{E_{2,0}} E_{1,r},
  E_{2,1},
  \ldots, E_{n,r}) \\
 &\mapsto \cdots \\
 &\mapsto
 (E_{0,0}, E_{1,0}, \ldots, E_{n,0},
  R_{E_{n,0}} \cdots R_{E_{1,0}} E_{0,1}, \ldots,
  R_{E_{n,0}} \cdots R_{E_{1,0}} E_{0,r}, \\
  &\ \hspace{4cm}
  R_{E_{n,0}} \cdots R_{E_{2,0}} E_{1,1}, \ldots,
  R_{E_{n,0}} E_{n-1,r},
  E_{n,1},
  \ldots, E_{n,r}) \\
  &=
 (E_{0,0}, E_{1,0}, \ldots, E_{n,0},
  F_{0,1}, \ldots, F_{0,r}, \ldots,
  F_{n,1}, \ldots, F_{n,r})
\end{align*}
where
\begin{align*}
 F_{i,j} = R_{E_{n,0}} R_{E_{n-1,0}} \cdots R_{E_{i+1,0}} E_{i,j}.
\end{align*}
Since
$
 \pi(A(i)) = \scO_Y(i) \otimes \rho_0 = E_{i,0}
$
for any $i \in \bZ$,
it follows that
$\Dbsing(\gr A)$ is equivalent
to the full triangulated subcategory of $D^b(\qgr A)$
generated by the exceptional collection
\begin{align*}
 (F_{0,1}, \ldots, F_{0,r},
  F_{1,1}, \ldots, F_{1,r},
   \ldots,
  F_{n,1}, \ldots, F_{n,r}).
\end{align*}
This proves Theorem \ref{th:main1}
in the case $d = 1$.

Now we discuss the case $d > 1$.
Since an exceptional is indecomposable and
the decomposition in \eqref{eq:dbcohY} is not only
semiorthogonal but orthogonal,
each exceptional object
in the full strong exceptional collection
$(E_{0,0}, \ldots, E_{n,r})$ on $Y$
belongs to one of orthogonal summands in \eqref{eq:dbcohY}.
It follows that the exceptional collection
in Corollary \ref{cr:PmodG} is divided
into $e$ copies of an exceptional collection,
each of which is pulled-back from $X$
and tensored with $\scO_Y(i)$ for $i = 0, \ldots, e-1$.
Let $(E_{i, j})_{(i,j) \in \Lambda}$ be the exceptional collection
generating the summand $p^*D^b \coh X$
in the orthogonal decomposition in \eqref{eq:dbcohY}.
Since $e$ divides $d$, the collection
$
 (\scO_Y, \scO_Y(d), \ldots, \scO_Y((a-1) d))
  = (p^* \scO_X, p^* \scO_X(d), \ldots, p^* \scO_X((a-1) d))
$
is a part of this collection.
On the other hand,
one has $\pi (B(i)) = \scO_X(d i)$ for any $i \in \bZ$
since $B$ is the $d$-th Veronese subring.
Now one can move these objects to the left by mutation,
and Theorem \ref{th:main1} follows from Theorem \ref{th:Orlov}
just as in the $d = 1$ case.

When $d = n+1$,
then $B$ is Gorenstein with parameter $1$,
and one does not need any mutation,
so that $\Dbsing(\gr B)$ has a full strong exceptional collection.

One can generalize the story
to the case with arbitrary weights
$\deg x_i = a_i$
and a finite subgroup $G \subset \SL_{n+1}(k)$
with a free action on $\bA^{n+1} \setminus \bszero$
commuting with the $\bG_m$-action.
The category $\qgr A$ is equivalent
to the category of coherent sheaves
on the weighted projective space
$\bP(a_1, \ldots, a_{n+1})$,
the Beilinson collection is given by
$
 (\scO, \scO(1), \ldots, \scO(a_1 + \cdots + a_{n+1})),
$
and the Gorenstein parameter
of the polynomial ring is
$
 a = a_1 + \cdots + a_{n+1}.
$
The case $d = a_1 + \cdots + a_{n+1}$ and $G = 1$
is discussed in \cite{Ueda_TCGCQS}.

\section{Crossed product algebras} \label{sc:crossed_product}

Let $A$ be an $\bN$-graded connected Gorenstein ring
with parameter $a > 0$
and $G$ be a finite group acting on $A$.
We assume that the characteristic of the base field $k$
does not divide the order of $G$.
The set of irreducible representations of $G$
will be denoted by
$
 \Irrep(G) = \{ \rho_0, \ldots, \rho_r \}
$
where $\rho_0$ is the trivial representation.

\begin{proof}[Proof of Theorem \ref{th:main3}]

We need to show the existence
of a full and faithful functor
$
 \Phi : \Dbsing(\gr A \rtimes G) \to D^b (\qgr A \rtimes G)
$
and a semiorthogonal decomposition
\begin{align*}
 D^b(\qgr A \rtimes G) = \langle \scO \otimes \rho_0, &\ldots,
  \scO \otimes \rho_r, \ldots, \\
  &\scO(a-1) \otimes \rho_0, \ldots, \scO(a-1) \otimes \rho_r,
  \Phi \Dbsing(\gr A \rtimes G) \rangle.
\end{align*}
%
%The proof is completely parallel to
%%Theorem \ref{th:Orlov} due to Orlov.
%\cite[Theorem 2.5.(i)]{Orlov_DCCSTCS},
%which we include for the reader's convenience.
%
Since $A$ is Gorenstein,
$A$ has finite injective dimension
as left and right module over itself.
It follows that $A \rtimes G$ also has
finite injective dimension
as left and right module over itself, and
one has mutually inverse equivalences
\begin{align*}
 D^{\phantom \circ} =
  \bR \Hom_{\phantom{(} A \rtimes G^{\phantom \circ}
   \phantom{)}}(\bullet, A \rtimes G)
  &: D^b(\gr^G A)^\circ \to D^b(\gr^G A^\circ), \\
 D^\circ = \bR \Hom_{(A \rtimes G)^\circ}(\bullet, A \rtimes G)
 & : D^b(\gr^G A^\circ)^\circ \to D^b(\gr^G A).
\end{align*}
of triangulated categories,
where $\bullet^\circ$ denotes the opposite rings and categories.

For an integer $i$,
let $\scS_{< i}$ be the full subcategory of $D^b(\grG A)$
consisting of complexes of torsion modules
concentrated in degrees less than $i$.
In other words, it is the full triangulated subcategory of $D^b(\grG A)$
generated by $k(e) \otimes \rho$
for $e < -i$ and $\rho \in \Irrep(G)$,
where $k(e) \otimes \rho$ is the $e$-shift
of the $A \rtimes G$-module
which is isomorphic to $\rho$ as a $G$-module
and annihilated by $A_+ = \bigoplus_{i=1}^\infty A_i$.
One can show
just as in \cite[Lemma 2.3]{Orlov_DCCSTCS}
that $\scS_{< i}$ is left admissible in $D^b(\grG A)$
and the left orthogonal is the derived category
$D^b(\grG A_{\ge i})$ of graded $G \rtimes G$ modules $M$
such that $M_p = 0$ for any $p < i$;
\begin{align} \label{eq:sd1}
 D^b(\grG A) = \la \scS_{< i}, D^b(\grG A_{\ge i}) \ra.
\end{align}
Let further $\scP_{< i}$ be the full subcategory of $D^b(\grG A)$
generated by projective modules $A(m) \otimes \rho$
for $m > - i$ and $\rho \in \Irrep(G)$.
One can also show
\begin{align} \label{eq:sd2}
 D^b(\grG A) = \la D^b(\grG A_{\ge i}), \scP_{<i} \ra
\end{align}
just as in \cite[Lemma 2.3]{Orlov_DCCSTCS}.
The proof of \cite[Lemma 2.4]{Orlov_DCCSTCS}
carries over verbatim to the $G$-equivariant case, and
gives weak semiorthogonal decompositions
\begin{align}
 D^b(\grG A_{\ge i}) &= \la \scD_i, \, \scS_{\ge i} \ra,
  \label{eq:sd3} \\
 D^b(\grG A_{\ge i}) &= \la \scP_{\ge i}, \, \scT_i \ra
  \label{eq:sd4}
\end{align}
where $\scD_i$ and $\scT_i$ are equivalent to
$D^b(\qgrG A)$ and $\Dbsing(\grG A)$ respectively.
\eqref{eq:sd1} and \eqref{eq:sd3} shows that
$\scS_{\ge i}$ is right admissible in $D^b(\grG A)$.
The functor $D$ takes the subcategory $\scS_{\ge i}(A)$
to the subcategory $\scS_{<-i-a+1}(A^\circ)$,
so that the right orthogonal $\scS_{\ge i}^\bot(A)$
is sent to the left orthogonal
$\!^\bot \scS_{<-i-a+1}(A^\circ)$.
The latter subcategory coincides with the right orthogonal
$\scP_{<-i-a+1}^\bot(A^\circ)$
by \eqref{eq:sd1} and \eqref{eq:sd2}.
The functor $D^\circ$ takes
the right orthogonal $\scP_{<-i-a+1}^\bot(A^\circ)$
to the left orthogonal $\!^\bot \scP_{\ge i+a}(A)$,
so that one has an equality
\begin{align} \label{eq:sd5}
 \scS_{\ge i}^\bot = \!^\bot \scP_{\ge i+a}
\end{align}
of subcategories of $D^b(\grG A)$.
%
%Now we can prove Theorem \ref{th:main3}:
One has
a weak semiorthogonal decomposition
\begin{align*}
 D^b(\grG A) = \la \scS_{<i}, \scD_i, \scS_{\ge i} \ra
\end{align*}
by \eqref{eq:sd1} and \eqref{eq:sd3},
which gives
\begin{align*}
 D^b(\grG A) = \la \scP_{\ge i+a}, \scS_{\ge i}, \scD_i \ra
\end{align*}
by \eqref{eq:sd5}.
Since Gorenstein parameter $a$ is positive,
the subcategory $\scP_{\ge i+a}$ is not only
right orthogonal but also left orthogonal to $\scS_{< i}$,
and one obtains a weak semiorthogonal decomposition
\begin{align} \label{eq:sd6}
 D^b(\grG A) = \la \scS_{\ge i}, \scP_{\ge i+a}, \scD_i \ra.
\end{align}
On the other hand,
\eqref{eq:sd1} and \eqref{eq:sd4} gives
a weak semiorthogonal decomposition
\begin{align} \label{eq:sd7}
 D^b(\grG A) = \la \scS_{\ge i}, \scP_{\ge i}, \scT_i \ra,
\end{align}
By combining 
\eqref{eq:sd6}, \eqref{eq:sd7} and
\begin{align*}
 \scP_{\ge i}
  = \langle \scP_{\ge i+a},
     A(-i-a+1) &\otimes \rho_0, \ldots, A(-i-a+1) \otimes \rho_r, \\
      & \ldots,
     A(-i) \otimes \rho_0, \ldots, A(-i) \otimes \rho_r
    \rangle,
\end{align*}
one obtains
\begin{align*}
 \scD_i
  = \langle
     A(-i-a+1) \otimes \rho_0, \ldots, \, &A(-i-a+1) \otimes \rho_r, \\
       &\ldots,
     A(-i) \otimes \rho_0, \ldots, A(-i) \otimes \rho_r,
     \scT_i
    \rangle,
\end{align*}
and Theorem \ref{th:main3} follows
by setting $i = -a+1$.
\end{proof}

%If $a < 0$, then there are full and faithful functors
%$
% \Psi_i : D^b(\qgr A) \to \Dbsing(\gr A \rtimes G)
%$
%and semiorthogonal decompositions
%$$
% \Dbsing(\gr A) = \la q A(-i-a+1), \ldots, q A(-i),
%  \Psi_i D^b(\qgr A) \ra
%$$
%where $\pi : D^b(\gr A) \to D^b(\qgr A)$ is the natural projection.

Let $A = \bigoplus_{i \in \bZ} R_{i d}$ be
the $d$-th Veronese ring of
$R = k[x_1, \ldots, x_{n+1}]$
for a divisor $d$ of $n+1$, and
$G$ be a finite subgroup of $\GL_{n+1}(k)$
whose order is not divisible by the characteristic of $k$.
Theorem \ref{th:main2} is an immediate consequence of
Theorem \ref{th:main3}:

\begin{proof}[Proof of Theorem \ref{th:main2}]
The graded ring $A$ is Gorenstein with parameter $a = (n+1)/d$,
and one has an equivalence
$$
 \qgrG A \cong \coh^G \bP^n
$$
of abelian categories.
The derived category $D^b \coh^G \bP^n$
has a full strong exceptional collection
\begin{align*}
 ( \scO_{\bP^n} \otimes \rho_0, \ldots, \scO_{\bP^n} \otimes \rho_r,
  \scO_{\bP^n}(1) \otimes \rho_0, \ldots,
  \scO_{\bP^n}(1) \otimes \rho_r, \ldots,
  \scO_{\bP^n}(n) \otimes \rho_0, \ldots, \scO_{\bP^n}(n) \otimes \rho_r
  ).
\end{align*}
Theorem \ref{th:main3} shows that
the full subcategory of $D^b \coh^G \bP^n$
generated by
\begin{align*}
 ( \scO_{\bP^n}(a) \otimes \rho_0, \ldots, \scO_{\bP^n}(a) \otimes \rho_r,
  \scO_{\bP^n}(a+1) \otimes \rho_0, &\ldots,
  \scO_{\bP^n}(a+1) \otimes \rho_r, \ldots, \\
  &\scO_{\bP^n}(n) \otimes \rho_0, \ldots, \scO_{\bP^n}(n) \otimes \rho_r
  )
\end{align*}
is equivalent to $\Dbsing(\grG A)$, and
Theorem \ref{th:main2} is proved.
\end{proof}

\section{Examples} \label{sc:example}

We discuss a few examples in this section.
Let us first consider the case
when $G \subset \SL_2(\bC)$ is the binary dihedral group
of type $D_4$.
The invariant subring $A = \bC[x_1, x_2]^G$
is generated by three elements $u$, $v$ and $w$
of degrees 4, 8 and 10
satisfying
$
 u^5 + u v^2 + w^2 = 0.
$
One has $\Irrep(G) = \{ \rho_0, \rho_1, \rho_2, \rho_3, \rho_4 \}$
and the quiver describing the total morphism algebra
of the full strong exceptional collection
$(\scO \otimes \rho_0, \ldots, \scO(1) \otimes \rho_4)$
is given as follows:
$$
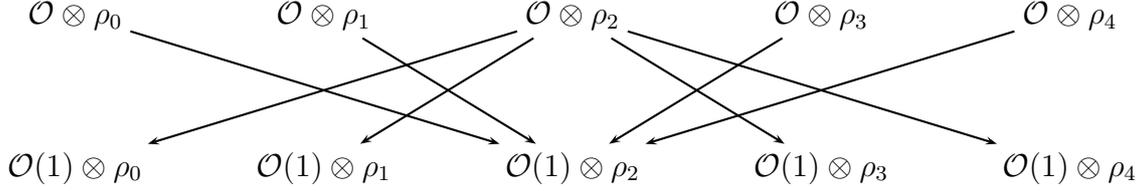
\begin{psmatrix}
 \scO \otimes \rho_0 & \scO \otimes \rho_1 & \scO \otimes \rho_2 &
 \scO \otimes \rho_3 & \scO \otimes \rho_4 \\
 \scO(1) \otimes \rho_0 & \scO(1) \otimes \rho_1 & \scO(1) \otimes \rho_2 &
 \scO(1) \otimes \rho_3 & \scO(1) \otimes \rho_4
\end{psmatrix} 
\psset{nodesep=3pt,arrows=->}
\ncline{1,1}{2,3}
\ncline{1,2}{2,3}
\ncline{1,4}{2,3}
\ncline{1,5}{2,3}
\ncline{1,3}{2,1}
\ncline{1,3}{2,2}
\ncline{1,3}{2,4}
\ncline{1,3}{2,5}
$$
Since the Gorenstein parameter of $A$ is two,
we have to remove $\scO \otimes \rho_0$ and
$\scO(1) \otimes \rho_0$ from the left.
The object $\scO \otimes \rho_0$ can be removed
without any mutation,
and when we remove $\scO(1) \otimes \rho_0$,
only $\scO \otimes \rho_2$ will be affected,
which will be turned into
$$
 R_{\scO \otimes \rho_2} \scO(1) \otimes \rho_0
  = \{ \scO \otimes \rho_2 \to \scO(1) \otimes \rho_0 \}.
$$
The resulting quiver is given as follows:
$$
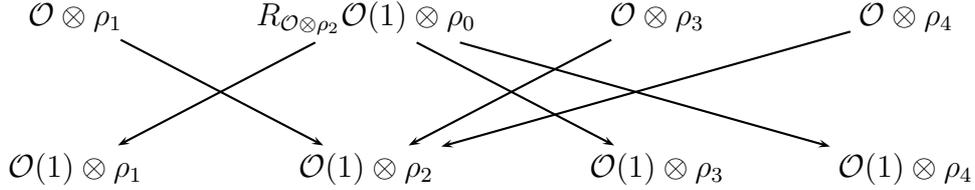
\begin{psmatrix}
 \scO \otimes \rho_1 &
 R_{\scO \otimes \rho_2} \scO(1) \otimes \rho_0 &
 \scO \otimes \rho_3 & \scO \otimes \rho_4 \\
 \scO(1) \otimes \rho_1 & \scO(1) \otimes \rho_2 &
 \scO(1) \otimes \rho_3 & \scO(1) \otimes \rho_4
\end{psmatrix} 
\psset{nodesep=3pt,arrows=->}
\ncline{1,1}{2,2}
\ncline{1,3}{2,2}
\ncline{1,4}{2,2}
\ncline{1,2}{2,1}
\ncline{1,2}{2,3}
\ncline{1,2}{2,4}
$$
The resulting full exceptional collection is strong in this case,
and the corresponding quiver is a disjoint union
of two Dynkin quivers of type $D_4$.

Now let us take a Veronese subring of $A$.
Since the Gorenstein parameter of $A$ is two,
only the second Veronese subring
$
 B = \bigoplus_{i \in \bZ} A_{2 i}
$
is Gorenstein,
which has Gorenstein parameter one.
Since $A$ has no odd components,
$B$ is isomorphic to $A$ as an algebra,
and only the grading is changed.
The stack $\Proj B = [ (\Spec B \setminus \bszero) / \bG_m]$
is a weighted projective line $\bX_{2,2,2}$
in the sense of Geigle and Lenzing
\cite{Geigle-Lenzing_WPC}
with three orbifold points of order $2$,
which is obtained from $\Proj A$
by the inverse root construction
(i.e. by removing the generic stabilizer).
It follows that $D^b \qgr A$ is equivalent
to the direct sum of two copies of $D^b \qgr B$,
and $D^b \qgr B$ is equivalent to the full subcategory
of $D^b \qgr A$ generated
by half of the full strong exceptional collection
in $D^b \qgr A$
shown below:
$$
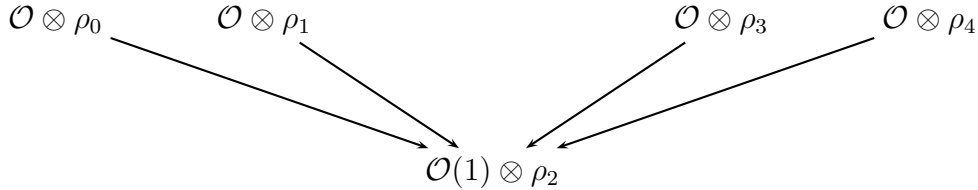
\begin{psmatrix}
 \scO \otimes \rho_0 & \scO \otimes \rho_1 & &
 \scO \otimes \rho_3 & \scO \otimes \rho_4 \\
 & & \scO(1) \otimes \rho_2 & &
\end{psmatrix} 
\psset{nodesep=3pt,arrows=->}
\ncline{1,1}{2,3}
\ncline{1,2}{2,3}
\ncline{1,4}{2,3}
\ncline{1,5}{2,3}
$$
Since the Gorenstein parameter of $B$ is one,
$\Dbsing(\gr B)$ is equivalent to the full subcategory
of $D^b(\qgr B)$
generated by the exceptional collection
obtained from the above collection
by removing $\scO \otimes \rho$,
which gives a Dynkin quiver of type $D_4$:
$$
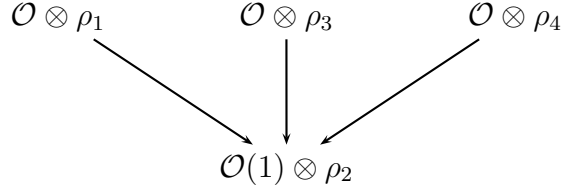
\begin{psmatrix}
 \scO \otimes \rho_1 & \scO \otimes \rho_3 & \scO \otimes \rho_4 \\
 & \scO(1) \otimes \rho_2 &
\end{psmatrix} 
\psset{nodesep=3pt,arrows=->}
\ncline{1,1}{2,2}
\ncline{1,2}{2,2}
\ncline{1,3}{2,2}
$$
On the other hand,
the crossed product algebra
$R \rtimes G$ with $R = \bC[x_1, x_2]$ is regular,
so that $\Dbsing(\gr^G R)$ is zero.
The graded stable derived category
$\Dbsing(\gr R^{(2)})$
of the second Veronese subring $R^{(2)} \rtimes G$
is equivalent to the full subcategory
of $D^b(\qgr^G R^{(2)}) \cong D^b \coh^G \bP^1$
generated by the strong exceptional collection
$$
 (\scO(1) \otimes \rho_0, \scO(1) \otimes \rho_1, \ldots,
   \scO(1) \otimes \rho_4)
$$
by Theorem \ref{th:main2},
which is just the direct sum of five copies
of the derived category of finite-dimensional vector spaces.

Next we consider the case when
$
 G = \la \exp(2 \pi \sqrt{-1}/3) \cdot \id_{\bA^3} \ra
$
is a cyclic subgroup of $\SL_3(k)$
of order three.
The total morphism algebra
of the full strong exceptional collection
$(\scO \otimes \rho_0, \ldots, \scO(2) \otimes \rho_2)$
in $D^b \coh^G \bP^2$
is given as follows:
$$
\begin{psmatrix}
 \scO \otimes \rho_0 & \scO(1) \otimes \rho_1 & \scO(2) \otimes \rho_2 \\
 \scO \otimes \rho_1 & \scO(1) \otimes \rho_2 & \scO(2) \otimes \rho_0 \\
 \scO \otimes \rho_2 & \scO(1) \otimes \rho_0 & \scO(2) \otimes \rho_1
\end{psmatrix} 
\psset{nodesep=3pt,arrows=->}
\ncline[offset=-3pt]{1,1}{1,2}
\ncline[offset=0pt]{1,1}{1,2}
\ncline[offset=3pt]{1,1}{1,2}
\ncline[offset=-3pt]{1,2}{1,3}
\ncline[offset=0pt]{1,2}{1,3}
\ncline[offset=3pt]{1,2}{1,3}
\ncline[offset=-3pt]{2,1}{2,2}
\ncline[offset=0pt]{2,1}{2,2}
\ncline[offset=3pt]{2,1}{2,2}
\ncline[offset=-3pt]{2,2}{2,3}
\ncline[offset=0pt]{2,2}{2,3}
\ncline[offset=3pt]{2,2}{2,3}
\ncline[offset=-3pt]{3,1}{3,2}
\ncline[offset=0pt]{3,1}{3,2}
\ncline[offset=3pt]{3,1}{3,2}
\ncline[offset=-3pt]{3,2}{3,3}
\ncline[offset=0pt]{3,2}{3,3}
\ncline[offset=3pt]{3,2}{3,3}
$$
Note that this is the disjoint union
of three copies of the Beilinson quiver for $\bP^2$.
The full exceptional collection in $\Dbsing(\gr A)$
is obtained from the above collection
by removing $\scO \otimes \rho_0$,
$\scO(1) \otimes \rho_0$ and
$\scO(2) \otimes \rho_0$.
To remove the second and the third object,
we can mutate the above collection as
$$
\begin{psmatrix}
 \scO \otimes \rho_0 & \scO(1) \otimes \rho_1 & \scO(2) \otimes \rho_2 \\
 \scO(2) \otimes \rho_0 & \scO(3) \otimes \rho_1 & \scO(4) \otimes \rho_2 \\
 \scO(1) \otimes \rho_0 & \scO(2) \otimes \rho_1 & \scO(3) \otimes \rho_2
\end{psmatrix} 
\psset{nodesep=3pt,arrows=->}
\ncline[offset=-3pt]{1,1}{1,2}
\ncline[offset=0pt]{1,1}{1,2}
\ncline[offset=3pt]{1,1}{1,2}
\ncline[offset=-3pt]{1,2}{1,3}
\ncline[offset=0pt]{1,2}{1,3}
\ncline[offset=3pt]{1,2}{1,3}
\ncline[offset=-3pt]{2,1}{2,2}
\ncline[offset=0pt]{2,1}{2,2}
\ncline[offset=3pt]{2,1}{2,2}
\ncline[offset=-3pt]{2,2}{2,3}
\ncline[offset=0pt]{2,2}{2,3}
\ncline[offset=3pt]{2,2}{2,3}
\ncline[offset=-3pt]{3,1}{3,2}
\ncline[offset=0pt]{3,1}{3,2}
\ncline[offset=3pt]{3,1}{3,2}
\ncline[offset=-3pt]{3,2}{3,3}
\ncline[offset=0pt]{3,2}{3,3}
\ncline[offset=3pt]{3,2}{3,3}
$$
so that the three objects
$\scO \otimes \rho_0$,
$\scO(1) \otimes \rho_0$ and
$\scO(2) \otimes \rho_0$ can safely be removed from the left
to obtain three copies of the generalized Kronecker quiver
$$
\begin{psmatrix}
 \bullet & \bullet
\end{psmatrix} 
\psset{nodesep=3pt,arrows=->}
\ncline[offset=-3pt]{1,1}{1,2}
\ncline[offset=0pt]{1,1}{1,2}
\ncline[offset=3pt]{1,1}{1,2}
$$
with three arrows.
On the other hand,
the third Veronese subring
$
 B = \bigoplus_{i \in \bZ} A_{3 i}
$
is Gorenstein with parameter one
and satisfies $\Proj B = \bP^3$,
so that $\Dbsing (\gr B)$ is equivalent
to the derived category of modules
over the generalized Kronecker quiver
with three arrows.
These results are in complete agreement
with the works of
Iyama and Yoshino \cite{Iyama-Yoshino_MTC},
Keller, Murfet and Van den Bergh
\cite{Keller-Murfet-Van_den_Bergh}, and
Iyama and Takahashi \cite{Iyama-Takahashi_TCTQS}.
The stable derived category of $R \rtimes G$
for the above $G$ and $R = \bC[x_1, x_2, x_3]$
is zero again,
and that of its third Veronese subring
$R^{(3)} \rtimes G$ is equivalent
to the full subcategory of
$D^b \qgr^G R^{(3)} \cong D^b \coh^G \bP^2$
generated by the strong exceptional collection
$$
 (\scO(1) \otimes \rho_0, \scO(1) \otimes \rho_1, \scO(1) \otimes \rho_2,
 \scO(2) \otimes \rho_0, \scO(2) \otimes \rho_1, \scO(2) \otimes \rho_2)
$$
which happens to be equivalent to $\Dbsing(\gr A)$ above.
%This is just a coincidence,
%as the example of type $D_4$ above shows.

Theorem \ref{th:main3} can be useful
also in other contexts.
%As an example,
%we discuss equivariant stable derived categories of
%invertible polynomials.
An integer $n \times n$ matrix $(a_{ij})_{i,j=1}^n$ defines
a polynomial
$$
 W = \sum_{i=1}^n x_1^{a_{i1}} \cdots x_n^{a_{in}},
$$
which is called {\em invertible}
if the origin is an isolated singularity.
They play a pivotal role
in {\em transposition mirror symmetry}
of Berglund and H\"{u}bsch
\cite{Berglund-Hubsch},
which attracts much attention recently.
See e.g.
\cite{%Borisov_BHMSVA,
%Chiodo-Ruan_quintic,
%Futaki-Ueda_BP,
Krawitz%,
%Takahashi_WPL
}
and references therein
for more on invertible polynomials
and mirror symmetry.

Any invertible polynomial is weighted homogeneous,
and the choice of a weight is unique up to multiplication
by a constant.
The quotient ring $A = k[x_1, \ldots, x_n] / (W)$ is Gorenstein
with parameter
$$
 a = \deg x_1 + \cdots + \deg x_n - \deg W.
$$
If $a$ is positive,
then for any group $G$ of symmetries of $W$,
one has a semiorthogonal decomposition
in Theorem \ref{th:main3}.
One can also prove
an analogue of \cite[Theorem 2.5.(ii),(iii)]{Orlov_DCCSTCS}
for $a \le 0$
just as in Theorem \ref{th:main3}.
A typical example is the case
when $G$ is a subgroup of the group
$$
 G^{\mathrm{max}}
  = \{ (\alpha_1, \ldots, \alpha_n) \in (k^{\times})^n \mid
  \alpha_1^{a_{11}} \cdots \alpha_n^{a_{1n}} = \cdots
   = \alpha_1^{a_{n1}} \cdots \alpha_n^{a_{nn}} = 1 \}
$$
of {\em maximal diagonal symmetries} of $W$,
but one can also deal with other cases
such as the action of the symmetric group $\frakS_n$
on the Fermat polynomial $W = x_1^m + \cdots + x_n^m$.

\bibliographystyle{amsalpha}
\bibliography{bibs}

\noindent
Kazushi Ueda

Department of Mathematics,
Graduate School of Science,
Osaka University,\\
Machikaneyama 1-1,
Toyonaka,
Osaka,
560-0043,
Japan.

{\em e-mail address}\ : \  kazushi@math.sci.osaka-u.ac.jp
\ \vspace{0mm} \\

\end{document}